\documentclass[a4paper, reqno, 12pt]{amsart}
\newif\iffinal
\finaltrue
\pdfoutput=1

\usepackage{float}
\usepackage[utf8]{inputenc}
\usepackage[T1]{fontenc}
\usepackage{amssymb}
\usepackage{mathtools}
\usepackage{a4wide}
\usepackage[section]{placeins}

\DeclareFontFamily{OT1}{pzc}{}
\DeclareFontShape{OT1}{pzc}{m}{it}{<-> s * [1.10] pzcmi7t}{}
\DeclareMathAlphabet{\mathpzc}{OT1}{pzc}{m}{it}

\usepackage{tikz}
\usetikzlibrary{decorations.pathreplacing, patterns, calc, positioning}

\usepackage{letltxmacro}
\usepackage{todonotes}
\LetLtxMacro\todonotestodo\todo
\renewcommand{\todo}[2][]{\todonotestodo[backgroundcolor=yellow, #1]{TODO: {#2}}}
\iffinal\renewcommand{\todo}[2][]{}\fi

\iffinal
\else
\usepackage[mark]{gitinfo2}
\renewcommand{\gitMark}{\jobname\,\textbullet{}\,\gitFirstTagDescribe\,\textbullet{}\,\gitAuthorName,\,\gitAuthorIsoDate}
\fi

\usepackage{hyperref}

\newtheorem{proposition}{Proposition}[section]
\newtheorem{theorem}[proposition]{Theorem}
\newtheorem{corollary}[proposition]{Corollary}

\newtheorem{lemma}[proposition]{Lemma}
\theoremstyle{definition}
\newtheorem{definition}[proposition]{Definition}

\newtheorem{example}[proposition]{Example}
\newtheorem{remark}[proposition]{Remark}

\DeclareMathOperator{\divclass}{\mathpzc{DivG}}
\DeclareMathOperator{\IIclass}{\mathpzc{InG}}
\DeclareMathOperator{\poclass}{\mathpzc{PoG}}
\DeclareMathOperator{\pdgclass}{\mathpzc{PDG}}
\DeclareMathOperator{\matclass}{\mathpzc{MatG}}
\DeclareMathOperator{\idemclass}{\mathpzc{IdemG}}
\DeclareMathOperator{\boolclass}{\mathpzc{BoolG}}
\DeclareMathOperator{\coneclass}{\mathpzc{-ConeG}}
\DeclareMathOperator{\diam}{diam}
\DeclareMathOperator{\pdg}{PDG}
\DeclareMathOperator{\girth}{g}

\newcommand{\first}[1]{\{1,\ldots, #1\}}

\DeclareMathOperator{\pos}{pos}
\DeclareMathOperator{\dir}{dir}
\DeclareMathOperator{\In}{In}
\DeclareMathOperator{\Div}{Div}
\DeclareMathOperator{\Idm}{Idm}
\DeclareMathOperator{\Mat}{MatG}

\DeclareMathOperator{\cone}{ConeG}

\newcommand{\calC}{\mathcal{C}}
\newcommand{\N}{\mathbb{N}}

\newcommand{\Z}{\mathbb{Z}}

\newcommand{\RamseyR}{\mathcal{R}}
\newcommand{\RamseyC}{\mathcal{R}_\calC}
\newcommand{\poR}{\RamseyR_{\poclass}}
\newcommand{\inR}{\RamseyR_{\IIclass}}
\newcommand{\pdgR}{\RamseyR_{\pdgclass}}
\newcommand{\divR}{\RamseyR_{\divclass}}
\newcommand{\matR}{\RamseyR_{\matclass}}
\newcommand{\idmR}{\RamseyR_{\idemclass}}
\newcommand{\booR}{\RamseyR_{\boolclass}}
\newcommand{\coneR}{\RamseyR_{k\!\!\;\coneclass}}
\newcommand{\coneZ}{\RamseyR_{3\!\!\;\coneclass}}
\newcommand{\units}[1]{U(#1)}

\author{Ayman Badawi}
\address{College of Arts and Sciences,
Department of Mathematics and Statistics \\ American University of Sharjah\\ Sharjah, UAE
}
\email{\href{mailto:abadawi@aus.edu}{abadawi@aus.edu}}

\author{Roswitha Rissner}
\address{Institut für Mathematik\\Alpen-Adria-Universität Klagenfurt\\
  Universitätsstraße 65-67\\9020 Klagenfurt am Wörthersee\\Austria}
\email{\href{mailto:roswitha.rissner@aau.at}{roswitha.rissner@aau.at}}

\title[Partial order graphs]{Ramsey numbers of partial order graphs
  (comparability graphs) and implications in ring theory}



\begin{document}

\begin{abstract}
  For a partially ordered set $(A, \le)$, let $G_A$ be the simple,
  undirected graph with vertex set $A$ such that two vertices
  $a \neq b\in A$ are adjacent if either $a \le b$ or $b \le a$.  We
  call $G_A$ the \emph{partial order graph} or {\emph{comparability
      graph}} of $A$. Further, we say that a graph $G$ is a partial
  order graph if there exists a partially ordered set $A$ such that
  $G = G_A$.  For a class $\calC$ of simple, undirected graphs and
  $n$, $m \ge 1$, we define the \emph{Ramsey number} $\RamseyC(n,m)$
  with respect to $\calC$ to be the minimal number of vertices $r$
  such that every induced subgraph of an arbitrary graph in $\calC$
  consisting of $r$ vertices contains either a complete $n$-clique
  $K_n$ or an independent set consisting of $m$ vertices.

  In this paper, we determine the Ramsey number with respect to some
  classes of partial order graphs. Furthermore, some implications of
  Ramsey numbers in ring theory are discussed.

\end{abstract}

\maketitle

\section{Introduction}
The Ramsey number $\RamseyR(n, m)$ gives the solution to the party
problem, which asks for the minimum number $\RamseyR(n, m)$ of guests
that must be invited so that at least $n$ will know each other or at
least $m$ will not know each other. In the language of graph theory,
the Ramsey number is the minimum number of vertices $v=\RamseyR(n,m)$
such that all undirected simple graphs of order $v$ contain a clique
of order $n$ or an independent set of order $m$. There exists a
considerable amount of literature on Ramsey numbers. For example,
Greenwood and Gleason~\cite{Greenwood-Gleason:1955:chromatic} showed
that $\RamseyR(3, 3) = 6$, $\RamseyR(3, 4) = 9$ and
$\RamseyR(3, 5) = 14$; Graver and
Yackel~\cite{Graver-Yackel:1968:graphramsey} proved that
$\RamseyR(3, 6) = 18$;
Kalbfleisch~\cite{Kalbfleisch:1966:chromaticramsey} computed that
$\RamseyR(3, 7) = 23$; McKay and Min~\cite{McKay-Min:1992;ramsey}
showed that $\RamseyR(3, 8) = 28$ and Grinstead and
Roberts~\cite{Grinstead-Roberts:1982:ramsey} determined that
$\RamseyR(3, 9) = 36$.

A summary of known results up to 1983 for $\RamseyR(n,m)$ is given in
Chung and Grinstead~\cite{CG}.  An up-to-date-list of the best
currently known bounds for generalized Ramsey numbers (multicolor
graph numbers), hypergraph Ramsey numbers, and many other types of
Ramsey numbers is maintained by Radziszowski~\cite{Ra}.

In this paper, we determine the Ramsey number of partial order
graphs. We want to point out that recently, a colleague kindly made us
aware that such graphs in literature are also known as {\emph
  comparability graph} and our result Theorem~\ref{thm:poramsey} is a
consequence of \cite[Theorem 6]{reff} (also see~\cite[Corollary
1]{reff}). However, our proof of Theorem~\ref{thm:poramsey} is
self-contained and it is completely different from the proof in
\cite{reff}. Our proof solely relies on the pigeon-hole principal. For
a partially ordered set $(A, \le)$, let $G_A$ be the simple,
undirected graph with vertex set $A$ such that two vertices
$a \neq b\in A$ are adjacent if either $a \le b$ or $b \le a$.  We
call $G_A$ the \emph{partial order graph} (\emph{comparability graph})
of $A$. In this paper, we will just use the name partial order
graph. Further, we say that a graph $G$ is a partial order graph if
there exists a partially ordered set $A$ such that $G = G_A$. For a
class $\calC$ of simple, undirected graphs and $n$, $m \ge 1$, we
define the \emph{Ramsey number} $\RamseyC(n,m)$ with respect to the
class $\calC$ to be the minimal number of vertices $r$ such that every
induced subgraph of an arbitrary graph in $\calC$ consisting of $r$
vertices contains either a complete $n$-clique $K_n$ or an independent
set consisting of $m$ vertices.

Next, we remind the readers of the graph theoretic definitions that
are used in this paper. We say that a graph $G$ is {\it connected} if
there is a path between any two distinct vertices of $G$. For vertices
$x$ and $y$ of $G$, we define {\it $d(x, y)$} to be the length of a
shortest path from $x$ to $y$ ($d(x, x) = 0$ and $d(x, y) = \infty$ if
there is no such path). The {\it diameter} of $G$ is
$\diam(G) = \sup\{d(x, y) \mid x\text{ and }y \text{ are vertices of }
G\}$. The {\it girth} of $G$, denoted by $\girth(G)$, is the length of
a shortest cycle in $G$ ($\girth(G) = \infty$ if $G$ contains no
cycles).  We denote the {\it complete graph} on $n$ vertices or
\emph{$n$-clique} by $K_n$ and the {\it complete bipartite} graph on
$m$ and $n$ vertices by $K_{m,n}$. The {\it clique number} $\omega(G)$
of $G$ is the largest positive integer $m$ such that $K_m$ is an
induced subgraph of $G$. The {\it chromatic number} of $G$, $\chi(G)$,
is the minimum number of colors needed to produce a proper coloring of
$G$ (that is, no two vertices that share an edge have the same
color). The {\it domination number} of $G$, $\gamma(G)$, is the
minimum size of a set $S$ of vertices of $G$ such that each vertex in
$G\setminus S$ is connected by an edge  to at least one vertex in $S$ by an
edge. An {\it independent vertex } set of $G$ is a subset of the
vertices such that no two vertices in the subset are connected by an
edge of $G$.  For a general reference for graph theory we refer to
Bollob\'{a}s' textbook~\cite{BB}.

In Section~\ref{sec:pographs} we show that the Ramsey number
$\poR(n,m)$ for the class  $\poclass$ of partial order graphs equals
$(n-1)(m-1)+1$, see~Theorem~\ref{thm:poramsey}. In
Section~\ref{sec:subclasses} we study subclasses of partial order
graphs that appear in the context of ring theory. Among other results,
we show that for the classes $\pdgclass$ of perfect divisor graphs,
$\divclass$ of divisibility graphs, $\IIclass$ of inclusion ideal
graphs, $\matclass$ of matrix graphs and $\idemclass$ of idempotents
graphs of rings, the respective Ramsey numbers equal to $\poR$,
see~Theorems~\ref{thm:pdg-ramsey}, \ref{thm:divisibility},
\ref{thm:inclusion}, \ref{matrices} and \ref{idm-graph}, respectively.
In Section~\ref{sec:unsymmetric} we a present a subclass of partial
ordered graphs with respect to which the Ramsey numbers are
non-symmetric.
	
Throughout this paper, $\Z$ and $\Z_n$ will denote the integers and
integers modulo $n$, respectively.  Moreover, for a ring $R$ we assume
that $1\neq 0$ holds, $R^\bullet=R \setminus \{0\}$ denotes the set of
non-zero elements of $R$ and $\units{R}$ denotes the group of units of
$R$.

\section{Ramsey numbers of partial order graphs}\label{sec:pographs}

\begin{definition}\label{def:pographs}
  \begin{enumerate}
  \item For a partially ordered set $(A, \le)$, let $G_A$ be the
    simple, undirected graph with vertex set $A$ such that two
    vertices $a \neq b\in A$ are adjacent if either $a \le b$ or
    $b \le a$.  We call $G_A$ the \emph{partial order graph} of
    $A$. Further, we say that $G$ is a partial order graph if there
    exists a partially ordered set $A$ such that $G = G_A$.  By
    $\poclass$ we denote the \emph{class of all partial order graphs}.
  \item For a class $\calC$ of simple, undirected graphs and $n$,
    $m \ge 1$, we set $\RamseyC(n,m)$ to be the minimal number of
    vertices $r$ such that every induced subgraph of an arbitrary
    graph in $\calC$ consisting of $r$ vertices contains either a
    complete $n$-clique $K_n$ or an independent set consisting of $m$
    vertices. We call $\RamseyC$ the \emph{Ramsey number with respect
      to the class~$\calC$}.
  \end{enumerate}
\end{definition}

\begin{theorem}\label{thm:poramsey}
  Let $n$, $m\ge 1$ ($n$, $m$ need not be distinct). Then for the
  Ramsey number $\poR$ with respect to the class $\poclass$ of
  partial order graphs, the following equality holds
  \begin{equation*}
    \poR(n,m) = \poR(m,n) = (n-1)(m-1)+1.
  \end{equation*}
\end{theorem}

\begin{proof}
   First, we prove that $\poR(n,m)> (n-1)(m-1)$. Let $A$ be a set of
  cardinality $(n-1)(m-1)$ and $A_1$, \ldots, $A_{n-1}$ an arbitrary
  partition of $A$ into $n-1$ subsets each of cardinality
  $m-1$. Further, for $a$, $b\in A$, we say $a \preceq b$ if and only
  if $a=b$ or $a\in A_i$ and $b\in A_j$ with $i<j$. Then $\preceq$ is
  a partial order on $A$ and the partial order graph $G_A$ is a
  complete $(n-1)$-partite graph in which each partition has $m-1$ independent
  vertices. It is easily verified that the clique number of $G_A$ is
  $n-1$ and at exactly  $m-1$ vertices of $G_A$ are independent.

  Let $G$ be a partial order graph and $H$ an induced subgraph. We
  show that if $H$ contains $(n-1)(m-1)+1$ vertices, then $H$ contains
  either an $n$-clique $K_n$  or an independent set of $m$ vertices.

  Let $G^{\dir}$ be the directed graph with the same vertex set as $G$
  such that $(a,b)$ is an edge if $a\neq b$ and $a \le b$. Then
  $H^{\dir}$ (the subgraph of $G^{\dir}$ induced by the vertices of
  $H$) contains a directed path of length $n$ if and only if $H$
  contains an $(n + 1)$-clique $K_{n+1}$.

  Note that $G^{\dir}$ does not contain a directed cycle. This
  allows us to define $\pos_H(a)$ to be the maximal length of a directed
  path in $H^{\dir}$ with endpoint $a$ for a vertex $a$ of $H$.

  It is easily seen, that $\pos_H(b) \le \pos_H(a) -1$ for every edge
  $(b,a)$ in $H^{\dir}$. In particular, if for two vertices $a$, $b$
  of $H$, $\pos_H(a) = \pos_H(b)$, then the two vertices are
  independent in $H$.

  Moreover, a straight-forward argument shows that $H$ contains an
  $n$-clique $K_n$  if and only if there exists a vertex $a$ in $H$ with
  $\pos_H(a) \geq n-1$.

  Now, assume that $H$ does not contain an $n$-clique $K_n$. This
  implies that $\pos_H(a) < n-1$ for all vertices $a$ in $H$. It then
  follows by the pigeonhole principle that among the $(n-1)(m-1) +1$
  vertices in $H$, there are at least $m$ vertices $a$ with
  $\pos_H(a) = k$ for some k, $0\le k\le n-2$. Therefore $H$ contains
  $m$ independent vertices.

  Since $(n-1)(m-1)+1$ is symmetric in $n$ and $m$, it further follows
  that $\poR(n, m) =\poR(m,n)$.
\end{proof}

\section{Subclasses of partial order graphs that appear in ring
  theory}\label{sec:subclasses}

In this section we discuss subclasses of partial order graphs that
appear in the context of ring theory. In particular, we focus on the
implications of Theorem~\ref{thm:poramsey}. Recall
for a class $\calC$ of graphs, $\RamseyC$ denotes the Ramsey number
with respect to $\calC$, cf.~Definition~\ref{def:pographs}.

\subsection{Perfect divisor graphs}\label{subsec:pdg}

\begin{definition}
  Let $R$ be a commutative ring, $n\in\N_{\ge 2}$ and
  $S = \{m_1, \ldots, m_n\} \subseteq R^\bullet\setminus \units{R}$ be
  a set of $n$ pairwise coprime non-zero non-units and
  $m = m_1m_2 \cdots m_n$. (Note that $m=0$ is possible.)
  \begin{enumerate}
  \item We say $d$ is a \emph{perfect divisor} of $m$ with respect to
    $S$ if $d \not = m$ and $d$ is a product of distinct elements of
    $S$.
  \item The \emph{perfect divisor graph} $\pdg(S)$ of $S$ is defined
    as the simple, undirected graph $(V,E)$ where
    $V= \{d \mid d \text{ perfect-divisor of } m\}$ is the vertex set
    and for two vertices $a \neq b\in V$, $(a,b) \in E$ if and only if
    $a \mid b$ or $b \mid a$.
  \item By $\pdgclass$ we denote the \emph{class of all perfect divisor
    graphs}.
  \end{enumerate}
\end{definition}

\begin{lemma}\label{lem:pd-order}
  Let $R$ be a commutative ring, $n\in\N_{\ge 2}$ and
  $S = \{m_1, \ldots, m_n\} \subseteq R^\bullet\setminus \units{R}$ be
  a set of $n$ pairwise coprime non-zero non-units and
  $m = m_1m_2 \cdots m_n$. Further, let
  \begin{equation*}
    V = \{d \mid d \text{ perfect divisor of } m \text{ with respect to
    }S\}
  \end{equation*}
  and define $\leq$ on $V$ such that for all $a$, $b \in V$, we have
  $a\leq b$ if and only if $a = b$ or $a \mid b$.

  Then $(V, \leq)$ is a partially ordered set of cardinality
  $|V| = 2^n - 2$ and $\pdg(S)$ is a partial order graph.
\end{lemma}

\begin{proof} The relation $\le$ clearly is reflexive and transitive,
  we prove that it is also antisymmetric.  Let $d \in V$ be a perfect
  divisor of $m$ with respect to $S$. Then $d = \prod_{j \in J} m_j$
  for $\emptyset \neq J \subseteq \first{n}$. We show that for every
  $1\le i \le m$, $m_i\mid d$ if and only if $i\in J$.

  Obviously if $j\in J$, then $m_i\mid d$. Let us assume that
  $i \in \first{n} \setminus J$. Then by hypothesis, for $j\in J$
  there are elements $a_j$ and $b_j \in R$ such that
  $a_jm_{j} + b_jm_i = 1$ holds. Hence
  \begin{equation*}
    1 = \prod_{j\in J} (a_jm_j + b_jm_i)
    = \left(\prod_{j\in J} a_jm_{j}\right) + cm_i = ad +cm_i
  \end{equation*}
  for some $a$, $c \in R$. Therefore, $d$ and $m_i$ are coprime elements of
  $R$ which in particular implies that $m_i \nmid d$.

  It follows that if $d_1$ and $d_2$ are distinct
  perfect divisors of $m$ and $d_1 \mid d_2$, then $d_2 \nmid
  d_1$. Thus $(V, \leq)$ is a partially ordered set.

  Moreover, it follows that the elements in $V$ correspond to the
  non-empty proper subset of $\first{n}$. Therefore, their number amounts to
  \begin{equation*}
    |V| = |\{\emptyset \neq J\subsetneq \first{n}\}| = \sum_{i=1}^{n-1}\binom{n}{i} = 2^n-2.
  \end{equation*}

\end{proof}

\begin{theorem}\label{thm:properties-pdg}
  Let $R$ be a commutative ring, $n\in\N_{\ge 2}$ and
  $S = \{m_1, \ldots, m_n\} \subseteq R^\bullet\setminus \units{R}$ be
  a set of $n$ pairwise coprime non-zero non-units,
  $m = m_1m_2 \cdots m_n$ and $\pdg(S)$ the perfect divisor graph of
  $m$ with respect to $S$.

  Then the following assertions hold:
  \begin{enumerate}
  \item\label{pdg:connected} $\pdg(S)$ is a connected graph if and only if $n\ge 3$.
  \item\label{pdg:diameter} If $n\ge 3$, then the diameter $\diam(\pdg(S)) = 3$.
  \item\label{pdg:dominating} The domination number of $\pdg(S)$ is
    equal $2$ if $n\geq 2$ and equal 1 if $n=1$.
  \item\label{pdg:partite} If $n \ge 3$, then the vertices in
    $P_k= \{\prod_{j\in J}m_j \mid |J| = k \}$ for $1 \le k \le n-1$
    are pairwise not connected by an edge. In particular, $\pdg(S)$ is
    an $(n-1)$-partite graph.
  \item\label{pdg:degree} If
    $a \in P_k = \{\prod_{j\in J}m_j \mid |J| = k \}$ for
    $1\le k\le n-1$, then $\deg(a) = 2^{k} + 2^{n-k} - 4$.
  \item\label{pdg:girth} If $n\ge 3$, then for the girth of $\pdg(S)$
    the following holds
    \begin{equation*}
      \girth(\pdg(S)) =
      \begin{cases}
        6 & n=3 \\
        3 & n\ge 4
      \end{cases}
    \end{equation*}
  \item\label{pdg:planar} $\pdg(S)$ is planar if and only if
    $n \in \{3,4\}$.
  \end{enumerate}
\end{theorem}

\begin{proof}
  \eqref{pdg:connected}: If $n=2$, then $V$ consists of 2 vertices
  $m_1$ and $m_2$ which are coprime and hence not connected.  Assume
  $n\ge 3$ and let $a = \prod_{j\in J}m_j$ and $b = \prod_{k\in K}m_k$
  be two distinct vertices of $\pdg(S)$. Suppose that $m_j = m_k$ for
  some $j\in J$ and $k\in K$. Then $a-m_j-b$ is a path of length 2
  from $a$ to $b$ if $m_j\neq a,b$ and $(a,b)$ is an edge
  otherwise. Suppose that $m_j \not = m_k$ for every $j \in J$ and
  $k \in K$. We show that $|J| \leq n - 2$ or $|K| \leq n -
  2$. Suppose that $|J| = |K| = n - 1$. Since $n \geq 3$ and
  $m_j \not = m_k$ for every $j \in J$ and $k \in K$, we conclude that
  $|\{m_j | j \in J\} \cup \{m_k \mid k \in K\}| = 2n - 2 > n$, a
  contradiction. Thus $|J| \leq n - 2$ or $|K| \leq n - 2$. Without
  loss of generality, we may assume that $|J| \leq n - 2$. Take
  arbitrary $k \in K$. Then, $a-am_k-m_k-b$ is a path of length $3$
  from $a$ to $b$ if $b\neq m_k$ and otherwise, $a-am_k-b$ is a path
  of length 2. Hence $\pdg(S)$ is connected which completes the proof
  of~\eqref{pdg:connected}.

  \eqref{pdg:diameter}: Suppose that $n \geq 3$. Then $\pdg(S)$ is
  connected by (1). Let $a, b$ be two distinct vertices of
  $\pdg(S)$. In light of the proof given in (1), we have
  $d(a, b) \leq 3$. Let $a = \prod_{j = 1}^{(n-1)}m_j$ and $b =
  m_n$. Then $a-m_1-m_1b-b$ is a shortest path in $\pdg(S)$ from $a$
  to $b$. Hence $d(a, b) = 3$. Thus $\diam(\pdg(S)) = 3$.

  For \eqref{pdg:dominating} observe, that every perfect divisor $d$ of
  $m$, is either divisible by $m_1$ or divides $m_2m_3\cdots
  m_n$. Hence, every vertex of $\pdg(S)$ is connected by an edge  to either one
  of these two vertices.

  \eqref{pdg:partite}: Let $1\le k \le n-1$ and $J$,
  $K \subseteq \first{n}$ with $|J| = |K| = k$. Set
  $a = \prod_{j\in J}m_j$ and $b = \prod_{k\in K}m_k$ be two different
  vertices of $\pdg(S)$ which implies $J\neq K$. Therefore, there
  exist $j \in J\setminus K$ and $k\in K\setminus J$.  In the proof
  of Lemma~\ref{lem:pd-order} we have shown that it now follows that
  $m_j \nmid b$ and $m_k \nmid a$. In particular, it follows that
  $a\nmid b$ and $b\nmid a$. Hence no two vertices in
  $\{\prod_{j\in J}m_j \mid |J| = k \}$ are connected by an edge.

  For~\eqref{pdg:degree}, let $a = \prod_{j\in J}m_j$ be perfect
  divisor of $m$ and set $k = |J|$. The perfect divisors of $m$ with
  respect to $S$ which divide $a$ which are connected by an edge to
  $a$ correspond to the non-empty, proper subsets of $J$ which are are
  $\sum_{i=1}^{k-1}\binom{k}{i} = 2^{k}-2$ many.  In addition, we need
  to count the number of perfect divisors of $m$ which are divisible
  by $a$. These are exactly the ones of the form $\prod_{k\in K}m_k$
  with $J \subsetneq K \subsetneq \first{n}$ of which there are
  $\sum_{i=1}^{n-k-1} \binom{n-k}{i} = 2^{n-k}-2$. Hence
  $\deg(a) = 2^k + 2^{n-k} -4$.

  \eqref{pdg:girth}: For $n = 3$, we can verify in
  Figure~\ref{fig:pdg-n3}, that there is cycle of length 6 and no
  shorter cycle.
  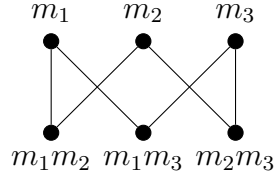
\begin{figure}[H]
  	\centering
  	\begin{tikzpicture}
  \tikzstyle{vertex}=[circle, fill=black, inner sep = 2,  draw,text=black]

  \node[style=vertex, label=above:$m_1$](m1) at (0,2) {};
  \node[style=vertex, label=above:$m_2$](m2)[right = 1cm of m1] {};
  \node[style=vertex, label=above:$m_3$](m3)[right = 1cm of m2] {};

  \node[style=vertex, label=below:$m_1m_2$](m12)[below = 1cm of  m1] {};
  \node[style=vertex, label=below:$m_1m_3$](m13)[right = 1cm of m12] {};
  \node[style=vertex, label=below:$m_2m_3$](m23)[right = 1cm of m13] {};

  \draw  (m1) -- (m12);
  \draw  (m1) -- (m13);
  \draw  (m2) -- (m12);
  \draw  (m2) -- (m23);
  \draw  (m3) -- (m13);
  \draw  (m3) -- (m23);
\end{tikzpicture}

  	\caption{Perfect divisor graph for $n=3$}
  	\label{fig:pdg-n3}
  \end{figure}

  If $n\ge 4$, then $m_1m_2m_3$ is a perfect divisor and the edges
  $(m_1, m_1m_2)$, $(m_1m_2, m_1m_2m_3)$ and $(m_1m_2m_3,m_1)$ form a
  cycle of length 3 which is the smallest possible length of a cycle
  in $\pdg(S)$.

  Finally, for~\eqref{pdg:planar}, it is easily verified that
  $\pdg(S)$ is planar if $n=3$, cf.~Figure~\ref{fig:pdg-n3}. Moreover,
  Figure~\ref{fig:pdg-n4} shows a planar arrangement of the edges of
  $\pdg(S)$ for $n=4$.

  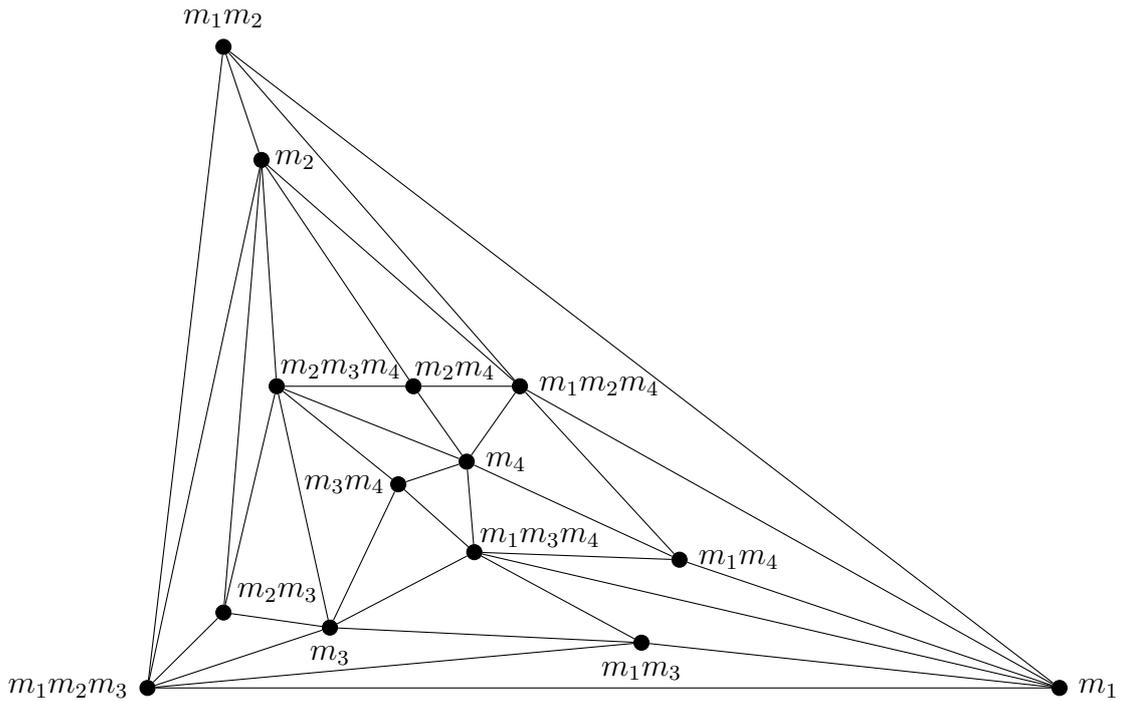
\begin{figure}[H]
  	\centering
  	\begin{tikzpicture}
  \tikzstyle{vertex}=[circle, fill=black, inner sep = 2,  draw,text=black]

  \node[style=vertex, label=above:$m_1m_2$] (12) at (1,8.5) {};       %

  \node[style=vertex, label={[label distance=-2pt] right:$m_2$}] (2) at (1.5,7) {};           %

  \node[style=vertex, label={[label distance=-6pt]25:$m_2m_3m_4$}] (234) at (1.7,4) {};   %
  \node[style=vertex, label={[label distance=-7pt]25:$m_2m_4$}] (24) at (3.5,4) {};       %
  \node[style=vertex, label=right:$m_1m_2m_4$] (124) at (4.9,4) {};   %

  \node[style=vertex, label={[label distance = -2pt] left:$m_3m_4$}] (34) at (3.3,2.7) {};       %
  \node[style=vertex, label=right:$m_4$] (4) at (4.2,3) {};           %

  \node[style=vertex, label={[label distance = -6pt]45:$m_1m_3m_4$}] (134) at (4.3,1.8) {}; %
  \node[style=vertex, label=right:$m_1m_4$] (14) at (7,1.7) {};       %

  \node[style=vertex, label= {[label distance=-2pt]3:$m_2m_3$}] (23) at (1,1) {};       %
  \node[style=vertex, label=below:$m_3$] (3) at (2.4,0.8) {};           %
  \node[style=vertex, label=below:$m_1m_3$] (13) at (6.5,0.6) {};       %

  \node[style=vertex, label=left:$m_1m_2m_3$] (123) at (0,0) {};    %
  \node[style=vertex, label=right:$m_1$] (1) at (12,0) {};           %

  \draw  (1) -- (12);
  \draw  (1) -- (13);
  \draw  (1) -- (14);
  \draw  (1) -- (123);
  \draw  (1) -- (124);
  \draw  (1) -- (134);

  \draw  (2) -- (12);
  \draw  (2) -- (23);
  \draw  (2) -- (24);
  \draw  (2) -- (123);
  \draw  (2) -- (124);
  \draw  (2) -- (234);

  \draw  (3) -- (13);
  \draw  (3) -- (23);
  \draw  (3) -- (34);
  \draw  (3) -- (123);
  \draw  (3) -- (134);
  \draw  (3) -- (234);

  \draw  (4) -- (14);
  \draw  (4) -- (24);
  \draw  (4) -- (34);
  \draw  (4) -- (124);
  \draw  (4) -- (134);
  \draw  (4) -- (234);

  \draw  (12) -- (123);
  \draw  (12) -- (124);

  \draw  (13) -- (123);
  \draw  (13) -- (134);

  \draw  (14) -- (124);
  \draw  (14) -- (134);

  \draw  (23) -- (123);
  \draw  (23) -- (234);

  \draw  (24) -- (124);
  \draw  (24) -- (234);

  \draw  (34) -- (134);
  \draw  (34) -- (234);

\end{tikzpicture} 

  	\caption{Perfect divisor graph for $n=4$ with planar
          arrangement of edges}
  	\label{fig:pdg-n4}
  \end{figure}

  If, however, $n\ge 5$, then Figure~\ref{fig:pdg-K33-minor} is a $K_{3, 3}$ subgraph of $\pdg(S)$, and hence $\pdg(S)$ is not a planar by Kuratowski’s Theorem on planar graphs.
  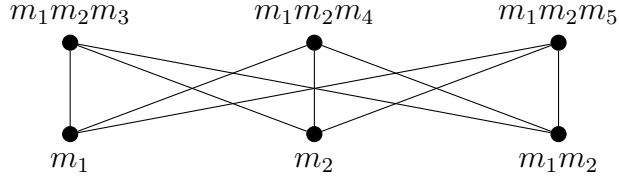
\begin{figure}[H]
    \centering
    \begin{tikzpicture}
  \tikzstyle{vertex}=[circle, fill=black, inner sep = 2,  draw,text=black]

  \node[style=vertex, label=above:{$m_1m_2m_3$}](m12) at (0,2) {};
  \node[style=vertex, label=above:{$m_1m_2m_4$}](m3)[right = 3cm of m12] {};
  \node[style=vertex, label=above:{$m_1m_2m_5$}](m4)[right = 3cm of m3] {};

  \node[style=vertex, label=below:{$m_1$}](m134)[below = 1cm of  m12] {};
  \node[style=vertex, label=below:{$m_2$}](m234)[right = 3cm of m134] {};
  \node[style=vertex, label=below:{$m_1m_2$}](m1234)[right = 3cm of m234] {};

  \draw  (m12) -- (m134);
  \draw  (m12) -- (m234);
  \draw  (m12) -- (m1234);
  \draw  (m3) -- (m134);
  \draw  (m3) -- (m234);
  \draw  (m3) -- (m1234);
  \draw  (m4) -- (m134);
  \draw  (m4) -- (m234);
  \draw  (m4) -- (m1234);
\end{tikzpicture}

    \caption{$\pdg(S)$ contains $K_{3,3}$ as a subgraph for $n\ge 5$.}
    \label{fig:pdg-K33-minor}
  \end{figure}

\end{proof}

Next, we compute the Ramsey number with respect to the class of
perfect divisor graphs. Note that $\pdgclass$ is a subclass of
$\poclass$ which immediately implies that $\pdgR(n,m) \le \poR(n,m)$
for all $n$, $m\ge 1$. We use Theorem~\ref{thm:properties-pdg} to show
that equality holds.

\begin{theorem}\label{thm:pdg-ramsey}
  Let $n$, $m\ge 1$. Then for the Ramsey number $\pdgR$ with respect
  to the class $\pdgclass$ of perfect divisor graphs the following
  holds
  \begin{equation*}
    \pdgR(n,m) = \poR(n,m) = (n-1)(m-1)+1.
  \end{equation*}
\end{theorem}

\begin{proof}
  We set $w = (n-1)(m-1)$ and show that $\pdgR(n, m) > w = (n-1)(m-1)$
  by giving an example of perfect divisor graph $G$ and an induced
  subgraph $H$ of $G$ with $w$ vertices which is a complete
  $(n-1)$-partite graph graph on $w$ vertices in which independent
  sets are of cardinality at most $m-1$.

  Let $R = \Z$ and let $S = \{p_1, p_2, \ldots, p_w\}$ be a set of $w$
  distinct positive prime numbers of $\Z$. We set
  $m = p_1p_2\cdots p_w$ and $G = \pdg(S)$.

  For each $1 \leq i \leq n - 1$, let $k_i =(i-1)(m-1)$ and we set
  $a_i = p_1p_2 \cdots p_{k_i}$ (where $a_1 = 1$) and
  \begin{equation*}
    A_i = \{a_i\,p_{k_i + 1}, \ldots, a_i\,p_{k_i + (m - 1)}\} .
  \end{equation*}
  Note that $A_{1} = \{p_1, \ldots, p_{m-1}\}$.

  Let $H$ be the subgraph of $G$ induced by the vertex set
  $A_1\cup A_2 \cup \cdots \cup A_{n-1}$. By construction, for each
  $1 \leq i \leq n - 1$, $|A_i| = m - 1$ holds and $A_i$ is contained
  in the partition $P_{k_i+1}$ of $G$,
  cf.~Theorem~\ref{thm:properties-pdg}.\ref{pdg:partite}. This implies
  that each $A_i$ is an independent vertex set of $H$ of cardinality
  $m-1$.

  Moreover, since $G$ is a $(w-1)$-partite graph and each $A_i$ is contained in  $P_{k_i+1}$, it follows that $H$
  is an $(n-1)$-partite graph (with partitioning
  $A_1 \cup A_2 \cup \cdots \cup A_{n-1}$). For an example of this
  construction with $m=5$ and $n=4$ see~Example~\ref{exa:div-graph}.

  Thus, no more than $m-1$ vertices of $H$ are independent and a
  straight-forward verification shows that the clique number of $H$ is
  at most $n-1$. Thus $\pdgR(n, m) > w$. Hence by
  Theorem~\ref{thm:poramsey}, we have
  $\pdgR(n, m) = \poR(n, m) = w + 1 = (n -1)(m - 1) + 1$.
\end{proof}

\begin{example}\label{exa:div-graph}
  We demonstrate the construction of the previous proof for the
  example $R = \Z$ with $n=4$ and $m=5$. That is, we construct a
  perfect divisor graph which has a complete $3$-partite graph $H$ as
  subgraph and each of the partitions of $H$ consist of $4$
  independent vertices.

  Let $w = (n-1)(m-1) = 12$ and we set
  $S = \{p_1, p_2, \ldots, p_{12}\}$. Next, let $n_i = (i-1)(m-1)$ for
  $1\le i \le 3$, that is, $n_1 = 0$, $n_2 = 4$ and $n_3 = 8$.  Then
  $a_1 = 1$, $a_2 = p_1p_2p_3 p_4$ and $a_3 = p_1p_2\cdots p_{8}$.

   We set
   \begin{align*}
     A_1 &= \{a_1p_1,a_1 p_2, a_1p_3, a_1p_4\} = \{p_1, p_2, p_3, p_4\} \\
     A_2 & = \{a_2p_5, a_2p_6, a_2p_7, a_2p_{8}\}    \\
         &= \{(p_1\cdots p_4)p_5, (p_1\cdots p_4)p_6, (p_1\cdots p_4)p_{7}, (p_1\cdots p_4)p_{8}\}    \\
     A_3 & = \{a_3p_{9}, a_3p_{10}, a_3p_{11}, a_3p_{12}\}    \\
         &= \{(p_1p_2\cdots p_{8})p_{9}, (p_1p_2\cdots p_{8})p_{10}, \ldots, (p_1p_2\cdots p_{8})p_{12}\}    \\
   \end{align*}
   The subgraph of $\pdg(S)$ induced by $A_1\cup A_2\cup A_3$ is a complete
   $3$-partite graph in which each partition has $4$ vertices that are
   independent, see~Figure~\ref{fig:div-graph}.

   \begin{figure}[h]
     \centering
     \begin{tikzpicture}
  \tikzstyle{vertex}=[circle, fill=black, inner sep = 2,  draw,text=black]

  \node[style=vertex, label=below:$p_1$](p1) at (1,0.5) {};
  \node[style=vertex, label=below:$p_2$](p2)[right = 1.3cm of p1.center] {};
  \node[style=vertex, label=below:$p_3$](p3)[right = 1.3cm of p2.center] {};
  \node[style=vertex, label=below:$p_4$](p4)[right = 1.3cm of p3.center] {};

  \node[style=vertex, label=left:$a_2p_5$](a1p5) [above = 1.5cm of p1.center]{};
  \node[style=vertex, label=left:$a_2p_6$](a1p6)[above = 1.5cm of p2.center] {};
  \node[style=vertex, label=right:$a_2p_7$](a1p7)[above = 1.5cm of p3.center] {};
  \node[style=vertex, label=right:$a_2p_8$](a1p8)[above = 1.5cm of p4.center] {};

  \node[style=vertex, label=above:$a_3p_9$](a2p9) [above = 3cm of p1.center]{};
  \node[style=vertex, label=above:$a_3p_{10}$](a2p10)[above = 3cm of p2.center] {};
  \node[style=vertex, label=above:$a_3p_{11}$](a2p11)[above = 3cm of p3.center] {};
  \node[style=vertex, label=above:$a_3p_{12}$](a2p12)[above = 3cm of p4.center] {};


  \draw  (p1) -- (a1p5);
  \draw  (p1) -- (a1p6);
  \draw  (p1) -- (a1p7);
  \draw  (p1) -- (a1p8);

  \draw  (p2) -- (a1p5);
  \draw  (p2) -- (a1p6);
  \draw  (p2) -- (a1p7);
  \draw  (p2) -- (a1p8);

  \draw  (p3) -- (a1p5);
  \draw  (p3) -- (a1p6);
  \draw  (p3) -- (a1p7);
  \draw  (p3) -- (a1p8);

  \draw  (p4) -- (a1p5);
  \draw  (p4) -- (a1p6);
  \draw  (p4) -- (a1p7);
  \draw  (p4) -- (a1p8);

  \draw  (a2p9) -- (a1p5);
  \draw  (a2p9) -- (a1p6);
  \draw  (a2p9) -- (a1p7);
  \draw  (a2p9) -- (a1p8);

  \draw  (a2p10) -- (a1p5);
  \draw  (a2p10) -- (a1p6);
  \draw  (a2p10) -- (a1p7);
  \draw  (a2p10) -- (a1p8);

  \draw  (a2p11) -- (a1p5);
  \draw  (a2p11) -- (a1p6);
  \draw  (a2p11) -- (a1p7);
  \draw  (a2p11) -- (a1p8);

  \draw  (a2p12) -- (a1p5);
  \draw  (a2p12) -- (a1p6);
  \draw  (a2p12) -- (a1p7);
  \draw  (a2p12) -- (a1p8);
  
\end{tikzpicture}

     \caption{Induced subgraph $H$ of $\pdg(\{p_1,\ldots, p_{12}\})$
       where, for better visibility, the edges between $A_1$ and $A_3$
       are ``hidden'' behind the edges between $A_1$ and $A_2$ and the
       edges between $A_2$ and $A_3$.}
     \label{fig:div-graph}
   \end{figure}
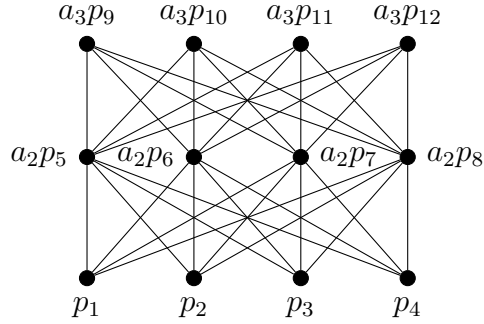

\end{example}

\subsection{The divisibility graph of a commutative ring}\label{subsec:divgraphs}

\begin{definition}\label{def:divgraph}
  Let $R$ be a commutative ring and $a$, $b$ be distinct elements of
  $R$.
  \begin{enumerate}
  \item If $a$ is a non-zero non-unit element of $R$, then we say $a$ is
    a \emph{proper} element of $R$.
  \label{def-item:proper}
  \item If $a \mid b$ (in $R$) and $b\nmid a$ (in $R$), then we write
    $a \mid\mid b$ .
  \item The \emph{divisibility graph} $\Div(R)$ of $R$, is the
    undirected simple graph whose vertex set consists of the proper
    elements of $R$ such that two vertices $a \neq b$ are adjacent if
    and only if $a \mid\mid b$ or $b \mid\mid a$.
  \end{enumerate}
\end{definition}

The following lemma can be verified by a straight-forward argument.
\begin{lemma}\label{lemma:divpo}
  Let $R$ be a commutative ring and let $V$ be the set of all proper
  elements of $R$ and define $\leq$ on $V$ such that for all
  $a$, $b \in V$, we have $a\leq b$ if and only if $a = b$ or
  $a \mid\mid b$.

  Then $(V, \leq)$ is a partially ordered set and the divisibility
  graph $\Div(R)$ of $R$ is a partial order graph.
\end{lemma}

By Lemma~\ref{lemma:divpo}, it is clear that
$\divR(n, m) \leq \poR(n, m)$ holds. However, since a perfect divisor
graph is an induced subgroup of a divisibility graph, it
follows from Theorem~\ref{thm:pdg-ramsey} that equality holds. We
conclude the following theorem.

\begin{theorem}\label{thm:divisibility}
  Let $n$, $m \geq 1$ be positive integers ($n$, $m$ need not be
  distinct). Then for the Ramsey number $\divR$ with respect to the
  class $\divclass$ of divisibility graphs the following holds

  \begin{equation*}
    \divR(n, m) = \poR(n, m) = (n - 1)(m - 1) + 1.
  \end{equation*}
\end{theorem}

Moreover, in view of Theorem~\ref{thm:divisibility}, we have the
following result.
\begin{corollary}\label{cor:elements}
  Let $n$, $m \geq 1$ be positive integers ($n$, $m$ need not be
  distinct), $k = (n-1)(m-1) + 1$, $R$ be a commutative ring and $S$
  be a subset of proper elements of $R$ such that $|S| \geq k$.

  Then one of the following assertions holds:
  \begin{enumerate}
  \item There are $n$ elements $a_1$, \ldots, $a_n\in S$ such that
    $a_1\mid\mid a_2 \mid\mid \cdots \mid\mid a_n$ (in $R$).
  \item There are $m$ pairwise distinct elements $b_1$, \ldots,
    $b_m\in S$ such that for all $1 \leq h \neq f \leq m$ either
    \begin{itemize}
    \item $b_h \nmid b_f$ or
    \item $b_h \mid b_f$ and $b_f \mid b_h$
    \end{itemize}
    holds.
  \end{enumerate}
\end{corollary}

\subsection{Inclusion ideal graphs of rings}\label{subsec:idealgraphs}

\begin{definition}\label{def:inclusion-graph}
  Let $R$ be a ring.
  \begin{enumerate}
  \item We call a left (right) ideal $I$ of $R$ \emph{non-trivial} if
    $I \not = \{0\}$ and $I \not = R$.
  \item The \emph{inclusion ideal graph} $\In(R)$ of $R$ is the
    (simple, undirected) graph whose vertex set is the set of
    non-trivial left ideals of $R$ and two distinct left ideals $I$,
    $J$ are adjacent if and only if $I \subset J$ or $J \subset I$
    (cf.~Akbari et.~al~\cite{Akbari-etal:2015:IIG}).
  \item By $\IIclass$, we denote the \emph{class of all inclusion
      ideal graphs}.
  \end{enumerate}
\end{definition}

\begin{remark}\label{rem:ideals-po}
  The set $V$ of all non-trivial left ideals of a ring $R$ together
  with the partial order $\subseteq$ induced by inclusion is a
  partially ordered set. Hence the inclusion graph $\In(R)$ of a ring
  $R$ is a partial order graph.
\end{remark}

By Remark~\ref{rem:ideals-po}, it is clear that
$\inR(n, m) \leq \poR(n, m)$. The reverse inequality can be seen from
the following argument. Let $G$ be the graph constructed in the
proof of Theorem~\ref{thm:pdg-ramsey}, that is, $G = \pdg(S)$ with
$S=\{p_1, \ldots, p_w\}$ is the set of $w = (n-1)(m-1)$ distinct
positive primes of $\Z$. Recall that $G$ contains a subgraph with
$(n-1)(m-1)$ vertices whose clique number is at most $n-1$ and in
which no more than $m-1$ vertices are independent. The graph $G$ is
graph-isomorphic to a subgraph of the inclusion ideal graph of $\Z$,
namely the subgraph induced by the principal ideals generated by the
elements in the vertex set of $G$. Since the inclusion ideal graph of
$\Z$ is contained in $\IIclass$, it follows that
$(n-1)(m-1) < \inR(n, m)$. Hence by Theorems~\ref{thm:poramsey} we
conclude the following theorem.

\begin{theorem}\label{thm:inclusion}
  Let $n$, $m \geq 1$ be positive integers ($n$, $m$ need not be
  distinct). Then for the Ramsey number $\inR$ with respect to the class $\IIclass$
  of inclusion ideal graphs the following holds
   \begin{equation*}
     \inR(n, m) = \poR(n, m) = (n - 1)(m - 1) + 1.
   \end{equation*}
\end{theorem}

In view of Theorem \ref{thm:inclusion}, we have the following result.
\begin{corollary}\label{ideals}
  Let $R$ be a ring, $n$, $m \geq 1$ be positive
  integers ($n$, $m$ need not be distinct) and
  $S \subseteq \{I \mid I \text{ is a non-trivial left ideal of }R\}$ such that
  $|S| \geq (n-1)(m-1) + 1$.

  Then one the following assertions hold:
  \begin{enumerate}
  \item There are $n$ pairwise distinct elements (non-trivial left ideals)
    $I_1$, \ldots, $I_n\in  S$ with
    $I_1 \subset I_2 \subset \cdots \subset I_n$.
  \item There are $m$ elements (non-trivial left ideals)
    $J_1$, \ldots , $J_m\in S$ such that $J_a \nsubseteq J_b$ for
    every $1 \leq a \neq b \leq m$.
  \end{enumerate}
\end{corollary}


\subsection{Matrix graphs over commutative rings}\label{subsec:matrixgraphs}

\begin{definition}
  Let $R$ be a commutative ring which is not a field and $j \ge 2$ an
  integer.
  \begin{enumerate}
  \item We denote by $R^{j\times j}$ the ring of all $j\times j$
    matrices with entries in $R$.
  \item Let
    $V = \{ A \in R^{j\times j} \mid \det(A) \text{ a proper element
      of } R\}$ be the set of all $j\times j$ matrices whose
    determinant is a proper element of $R$,
    cf.~Definition~\ref{def:divgraph}.\ref{def-item:proper}.  We
    define the \emph{matrix graph} $\Mat(R)$ of $R$ to be the
    undirected simple graph with $V$ as its vertex set and two
    distinct vertices $A$, $B \in V$ are adjacent if and only if
    $\det(A) \mid\mid \det(B)$ or $\det(A) \mid\mid \det(B)$.

  \item By $\matclass$ we denote the \emph{class of all matrix
      graphs}.
  \end{enumerate}

\end{definition}

\begin{lemma}
  Let $R$ be a commutative ring which is not a field, $j \ge 2$ an
  integer and
  \begin{equation*}
    V = \{ A \in R^{j\times j} \mid \det(A) \text{ is a proper element of } R\}.
  \end{equation*}
  Define $\leq$ on $V$ such that for all $A$, $B\in V$, we have
  $A\leq B$ if and only if $A = B$ or $\det(A) \mid\mid \det(B)$.

  Then $(V, \leq)$ is a partially ordered set and the graph $\Mat(R)$
  is a partial order graph.
\end{lemma}

By Theorem \ref{thm:poramsey}, it is clear that
$\matR(n, m) \leq \poR(n, m)$. We prove next that equality holds.

\begin{theorem}\label{matrices}
  Let $n$, $m \geq 1$ be positive integers ($n$, $m$ need not be
  distinct). Then for the Ramsey number $\matR$ with respect to the
  class $\matclass$ of matrix graphs the following holds
  \begin{equation*}
    \matR(n, m) = \poR(n, m) = (n - 1)(m - 1) + 1 .
  \end{equation*}
\end{theorem}

\begin{proof}
  Let $R = \Z$ and $j \geq 2$ and set $w = (n-1)(m-1) \geq 1$.
  Further, let $p_1$, $p_2$, \ldots, $p_w$ be distinct positive prime
  numbers of $\Z$ and choose $X_i \in R^{j\times j}$ with
  $\det(X_i) = p_i$ for $1 \leq i \leq w$.

  We construct a matrix graph $\Mat(R)$ which has a complete
  $(n-1)$-partite subgraph $H$ in which each partition has $m-1$
  vertices. The construction is analogous to the one in the proof of
  Theorem~\ref{thm:pdg-ramsey}.

  For each $1 \leq i \leq n - 1$, let $k_i = (i-1)(m - 1)$,
  $q_i = X_1X_2 \cdots X_{n_i}$ (hence $q_1 = I_j$ the identity matrix $j\times j$) and
  \begin{equation*}
    A_i = \{q_iX_{n_i + 1}, \ldots, q_iX_{k_i + (m - 1)}\}.
 \end{equation*}
 Note that $A_{1} = \{X_1, \ldots, X_{m-1}\}$. Since
 $\det(q_iX_{n_i+j}) = p_1\ldots p_{n_i}p_{n_i+j}$, it follows that the
 elements of $A_i$ are pairwise distinct and $|A_i| = m - 1$
for  $1 \leq i \leq n - 1$.

 Let $S = A_1\cup A_2 \cup \cdots \cup A_{n-1}$ and set
 $G = \Mat(\Z)$. Then for each $i$, the vertices in $A_i$ are
 independent. However, there are edges between all vertices of two
 distinct sets $A_i$ and $A_j$ with $i\neq j$.  Therefore, $G$ is a
 complete $(n-1)$-partite graph in which each partition has $m-1$
 vertices that are independent. Thus at most $m-1$ vertices of $G$ are
 independent. It is easily verified that the clique number of $G$ is
 $n-1$. It follows that $\matR(n, m) > w$ and together with
 Theorem~\ref{thm:poramsey} we conclude
 $\matR(n, m) = \poR(n, m) = w + 1 = (n -1)(m - 1) + 1$.
\end{proof}

\begin{corollary}\label{det-matrices}
  Let $R$ be a commutative ring, $j \geq 2$, $n$,
  $m \geq 1$ be positive integers ($n$, $m$ need not be distinct) and
  $S \subseteq \{X \in D \mid \det(X) \text{ is a proper element of }
  R\}$ such that $|S| \geq (n-1)(m-1) + 1$.

  Then one of the following assertions hold:
  \begin{enumerate}
  \item There are $n$ matrices $X_1$, \ldots, $X_n\in S$ such that
    $\det(X_1) \mid\mid \det(X_2) \mid\mid \cdots \mid\mid \det(X_n)$
    (in $R$).
  \item There are $m$ pairwise distinct matrices $Y_1$, \ldots,
    $Y_m\in S$, such that for all $1 \leq h \neq f \leq m$.
    \begin{itemize}
    \item $\det(Y_h) \nmid \det(Y_f)$ or
    \item $\det(Y_h) \mid \det(Y_f)$ and $\det(Y_f) \mid \det(Y_h)$
    \end{itemize}
    holds.
  \end{enumerate}
\end{corollary}


\subsection{Idempotents graphs of commutative
  rings}\label{subsec:idempotents}

\begin{definition}
  Let $R$ be a commutative ring.
  \begin{enumerate}
  \item We call $a\in R$ \emph{idempotent} if $a^2 = a$.
  \item We define the \emph{idempotents graph} $\Idm(R)$ of $R$ to be
    the undirected simple graph with the set of idempotents of $R$ as its
    vertex set and two distinct vertices $a$, $b$ are adjacent if and
    only if $a \mid b$ or $b \mid a$.

  \item By $\idemclass$ we denote the \emph{class of all idempotents graphs}.
  \end{enumerate}
\end{definition}

First, we show that the divisibility relation is a partial order on
the set of idempotent elements of $R$.

\begin{lemma}\label{ee}
  Let $R$ be a commutative ring and let $V$ be the
  set of all idempotent elements of $R$. We define $\leq$ on $V$ such
  that for all $a$, $b \in V$, we have $a\leq b$ if and only if
  $a \mid b$.

  Then $(V, \leq)$ is a partially ordered set and the graph $\Idm(R)$
  is a partial order graph.
\end{lemma}

\begin{proof}
  Clearly, $\leq$ is reflexive and transitive. Suppose that $a\mid b$
  and $b \mid a$ (in $R$), that is, $a = bx$ and $b = ay$ for some
  $x$, $y \in R$. Then, since $a$ and $b$ are idempotent, we can
  conclude that
  \begin{align*}
    a - ba &= (1-b)a = (1-b)bx = bx - b^2x = bx - bx = 0 \text{ and } \\
    b - ab &= (1 - a)b = (1 - a)ay = ay-a^2y = ay -ay = 0
  \end{align*}
  and hence $a = ba = ab = b$ which implies that $\leq$ is
  anti-symmetric.
\end{proof}

By Lemma~\ref{ee}, it is clear that
$\idmR(n, m) \leq \poR(n, m)$. Next, we show that
$\idmR(n, m) = \poR(n, m)$. We start with the following lemma.

\begin{lemma}\label{maxidm}
  Let $R$ be a commutative ring and $E$ be a set of
  $w\ge 3$ distinct non-trivial idempotents of $R$ such that $eR$ is a
  maximal ideal of $R$ for every $e\in E$. Let $x = f_1f_2\cdots f_k$
  and $y = b_1b_2\cdots b_j$ such that
  $f_1$, \ldots, $f_k$, $b_1$, \ldots, $b_j \in E$ and $2 \leq k, j < w$.

  Then
  \begin{enumerate}
  \item $x \not = 0$.
  \item $x = y$ if and only if $\{f_1,\ldots, f_k\} = \{b_1,\ldots, b_j\}$.
  \end{enumerate}
\end{lemma}

\begin{proof}
  (i) Since $e_1$, \ldots, $e_w$ are distinct non-trivial idempotents
  of $R$ and each $e_iR$ is a maximal ideal of $R$, $1 \leq i \leq w$,
  by Lemma~\ref{ee} we conclude that $e_1R$, \ldots, $e_wR$ are
  distinct maximal ideals of $R$.  Since $k < w$, there exists a
  maximal ideal $dR$ for some $d \in E$ such that
  $x = f_1f_2\cdots f_k \notin dR$ (note that each $f_iR$ is a maximal
  ideal of $R$). Thus $x \not = 0$.

  (ii) We may assume that $f_1 \not = b_i$ for every
  $1 \leq i \leq j$. Hence $x \in f_1R$ but $y \notin f_1R$ and thus
  $x \not = y$. Since all $f_i$ and $b_i$ are idempotent elements,
  multiplicities have no impact which makes the other implication
  obvious.
\end{proof}

\begin{theorem}\label{idm-graph}
  Let $n$, $m \geq 1$ be positive integers ($n$, $m$ need not be
  distinct). Then for the Ramsey number $\idmR(n,m)$ with respect to
  the class of idempotents graphs the following holds
  \begin{equation*}
    \idmR(n, m) = \poR(n, m) = (n - 1)(m - 1) + 1.
  \end{equation*}
\end{theorem}

\begin{proof}
  We set $w = (n-1)(m-1) \geq 1$ and show that $\idemclass$ contains
  an $(n-1)$-partite graph in which each partition consists of $m-1$
  independent vertices. For this purpose, set $R = \prod_{i=1}^w\Z_2$.
  It is clear that $R$ has exactly $w$ distinct maximal ideals, say
  $M_1$, \ldots , $M_w$, and each $M_i = p_iR$, $1 \leq i \leq w$ for
  idempotent $p_i$ of $R$. We set $E = \{p_1, p_2, \ldots,
  p_w\}$. Note that $|E| = w$ since $p_1$, $p_2$, \ldots, $p_w$ are
  pairwise distinct.

  For each $1 \leq i \leq n - 1$, let $n_i = (i-1)(m - 1)$,
  $a_i = p_1p_2 \cdots p_{n_i}$ (hence $a_1 = 1$) and
  $A_i = \{a_ip_{n_i + 1}, \ldots, a_ip_{n_i + (m - 1)}\}$. Note that
  $A_{1} = \{p_1, \ldots, p_{m-1}\}$.

  By construction of each $A_i$ and in light of Lemma~\ref{maxidm},
  for each $1 \leq i \leq n - 1$, we have $|A_i| = m - 1$ and the
  vertices of $A_i$ are independent.  Let $H$ be the subgraph of
  $\Idm(R)$ which is induced by
  $A_1\cup A_2 \cup \cdots \cup A_{n-1}$.

  By construction of $H$ and Lemma~\ref{maxidm}, we conclude that $H$
  is a complete $(n-1)$-partite graph in which each partition has
  $m-1$ vertices that are independent. Thus $H$ has exactly $m-1$
  vertices that are independent. It is easily verified that the clique
  number of $H$ is $n-1$. Thus $\idmR(n, m) > w$. Hence by
  Theorem~\ref{thm:poramsey}, we have
  $\idmR(n, m) = \poR(n, m) = w + 1 = (n -1)(m - 1) + 1$.
\end{proof}

\begin{remark}
  Observe that the ring $R =\prod_{i=1}^w\Z_2$ in the proof of
  Theorem~\ref{idm-graph} is a finite boolean ring. Let $\boolclass$
  denote the subclass of $\idemclass$ consisting of all idempotents
  graphs of boolean rings.

  In view of the proof of Theorem~\ref{idm-graph}, we conclude that
  $\booR(n, m) = \idmR(n, m)$. Thus we state this result without a
  proof.
\end{remark}

 \begin{theorem}\label{boolean}
   Let $n, m \geq 1$ be positive integers ($n, m$ need not be
   distinct).

   Then $\booR(n, m) = \idmR(n, m) = \poR(n, m) = (n - 1)(m - 1) + 1$.
 \end{theorem}

 In view of Theorem~\ref{idm-graph}, we have the following result.
\begin{corollary}\label{idempotents}
  Let $n$, $m \geq 1$ be positive integers ($n$, $m$ need not be
  distinct), $k = (n-1)(m-1) + 1$ and $A$ be a subset of idempotent
  elements of $R$ such that $|A| \geq k$.

  Then one of the following assertions hold
  \begin{enumerate}
  \item There are $n$ pairwise distinct elements (distinct idempotents) $a_1$,
    \ldots, $a_n\in A$ such that $a_1\mid a_2 \mid \cdots \mid a_n$
    (in $R$).
  \item There are $m$ pairwise distinct elements (distinct
    idempotents) $b_1$, \ldots, $b_m\in A$ such that $b_h \nmid b_f$
    (in $R$) for all $1 \leq h \neq f \leq m$.
  \end{enumerate}
\end{corollary}

\section{An example class $\mathcal{C}$ of partial order graphs with
  $\RamseyR_{\mathcal{C}}(n,m) \not =
  \RamseyR_{\mathcal{C}}(m,n)$}\label{sec:unsymmetric}

In this section, we present a subclass $\calC$ of $\pdgclass$ with
respect to which the Ramsey numbers $\RamseyC$ are non-symmetric in
$m$ and $n$. We recall the following definition \cite{KT2}.
\begin{definition}
	A subset $S$ of a ring R is called a \emph{positive semi-cone} of $R$ if $S$  satisfies the following conditions:
	\begin{enumerate}
		\item $S\cap (-S) = \{0\}$.
		\item $S + S \subseteq S$.
		\item $S\cdot S \subseteq S$.
	\end{enumerate}
	If $S$ satisfies the above conditions and $S\cup (-S) = R$, then $S$ is called a \emph{positive cone} of $R$ \cite{KT}.
\end{definition}
	
For a positive semi-cone $S$ of $R$, define $\leq_S$ on $R$ such that
for all $a$, $b\in R$, we have $a\leq_S b$ if and only if
$b - a \in S$. Then $(R, \leq_S)$ is a partially ordered set. We
define the \emph{$S$-positive semi-cone graph} $\cone_S(R)$ of $R$ to
be the simple, undirected graph with vertex set $R$ such that two
vertices $a$, $b$ are connected by an edge if and only if $b -a \in S$
or $a - b \in S$. Then $\cone_S(R)$ is a partial ordered graph.

\begin{definition}  For $k \ge 2$, let $P_k = \{0, k, 2k, 3k, \ldots\} = k\N_0$.
  \begin{enumerate}
  \item For $a$, $b\in \Z$ we define $a \le_k b$ if and only if
    $b - a \in P_k$.
  \item We define the \emph{$k$-positive semi-cone graph} $\cone_k(\Z)$ of $\Z$ to be
    the simple, undirected graph with vertex set $\Z$ such that two
    vertices $a$, $b \in R$ are connected by an edge if and only if $|a-b|\in P_k$.
    \item For every positive integer $k \geq 2$, let $\coneR(n,m)$ to be the minimal number of
    vertices $r$ such that every induced subgraph of the 
    partial order graph $\cone_k(\Z)$ consisting of $r$ vertices contains either a
    complete $n$-clique $K_n$ or an independent set consisting of $m$
    vertices. 
   \end{enumerate}
\end{definition}

\begin{remark}
	  \begin{enumerate}
  
  \item For every $k\geq 2$, $P_k$ is a positive semi-cone subset of $\Z$ that is not a positive cone of $\Z$. The relation $a \le_k b$ if and only if $b - a\in P_k$ is a partial order on $\Z$ and $\cone_k(\Z)$ is a partial order graph.
  \item For every $k\geq 2$, then two vertices $a$, $b$ of $\cone_k(\Z)$ are connected by an edge
    if and only if $a\equiv b \ (\mathrm{mod} \ \ k)$.
  \end{enumerate}
\end{remark}

For each $k \geq 2$, the following theorem shows that $\coneR(n, m)$ is not always
symmetric in $m$ and $n$.

\begin{theorem}\label{fun} Let $k \geq 2$, $n$, $m \geq 1$ be positive
  integers ($n$, $m$ need not be distinct). Then
  \begin{enumerate}
  \item\label{fun-1} If $1 \leq m \leq k+1$, then
    \begin{equation*}
      \coneR(n, m) = (n - 1)(m - 1) + 1.
    \end{equation*}
    In particular, if $1 \leq n, m \leq k+1$, then
    $\coneR(n, m) = \coneR(m, n) = (n -1)(m - 1) + 1$ is symmetric in
    $n$ and $m$.
  \item\label{fun-2} If $m > k+1$, then
    \begin{equation*}
      \coneR(n, m) = \coneR(n, k+1) = (n - 1)k + 1
    \end{equation*}
    only depends on the first argument $n$. In particular, assume that $n \not = m$. If $n > k+1$ or $m > k+1$, then
    $\coneR(n, m) \not = \coneR(m, n)$.
\end{enumerate}
\end{theorem}

\begin{proof}
  \eqref{fun-1}: For $n=1$ or $m=1$, the assertion immediately
  follows, so we assume $n\ge 2$ and $2 \le m\le k + 1$. For each
  $1 \leq i \leq m - 1$, let

  \begin{equation*}
    A_i = \{k + i, 2k + i, \ldots, (n-1)k + i\}
  \end{equation*}
  By construction, each $A_i$ contains $n-1$ distinct elements $a$
  with $a-i \in P_k$. Therefore for $a\neq b\in A_i$, either
  $b-a\in P_k$ or $a-b\in P_k$ and hence each $A_i$ induces a complete
  subgraph of $\cone_k(\Z)$ with exactly $n -1$ vertices. Moreover, since
  $m-1 \le k$, for $a\in A_i$ and $b\in A_j$ with
  $1 \le i \neq j \le m-1$, then $a\not\equiv b \ \ (\mathrm{mod} \ \ k)$ and
  therefore $a$ and $b$ are not connected by an edge.

  Let $H$ be the subgraph of $\cone_k(\Z)$ which is induced by the vertex
  set $A_1 \cup \cdots \cup A_{m-1}$.  Then $H$ is disjoint union of
  $m-1$ $(n-1)$-cliques and hence does neither contain an $n$-clique
  nor an independent set of cardinality $m$ which implies that
  $\coneR(n, m) > (n-1)(m-1)$. It now follows from
  Theorem~\ref{thm:poramsey} that $\coneR(n, m) = (n -1)( m- 1) + 1$.

  The symmetry assertion follows immediately from this if, moreover,
  $1 \le n \le k+1$ holds.

  \eqref{fun-2}: Recall that two vertices $a$, $b$ of $\cone_k(\Z)$ are
  connected by an edge if and only if $a \equiv b \ (\mathrm{mod} \ \ k)$. Therefore, a maximal independent subset has cardinality $k$
  (the number of residue classes modulo $k$).  Thus if $m \ge k+1$, then
  $\cone_k(\Z)$ cannot contain an independent set with $m$ distinct
  vertices. Therefore, for all $m\ge k+1$, the equality
  \begin{equation*}
    \coneR(n,m) = \coneR(n,k+1)
  \end{equation*}
  holds and the assertion now follows from~\eqref{fun-1}.
\end{proof}

In view of Theorem~\ref{fun}, we have the following result.
\begin{corollary}\label{funZ}
  Let $k \geq 2$ and $n$, $m\ge 1$ be positive integers ($n$, $m$ need not
  be distinct) and $A$ be a subset of $\Z$. Then
\begin{enumerate}
\item If $2 \leq m \leq k + 1$ and $|A| > (n - 1)(m - 1)$, then there are
  at least $n$ pairwise distinct elements $a_1$, \ldots, $a_n \in A$
  such that $a_1 \equiv \cdots \equiv a_n \ (\mathrm{mod} \ \  k)$ or there at least
  $m$ elements $b_1$, \ldots, $b_m \in A$ such that
  $b_i \not \equiv b_j \ (\mathrm{mod} \ \  k)$ for all $1 \leq i \neq j \leq m$.

\item If $m > k + 1$ and $|A| > (n - 1)k$, then there are at least $n$
  pairwise distinct elements of $A$, say $a_1$, \ldots, $a_n$ such
  that $a_1 \equiv \cdots \equiv a_n \ (\mathrm{mod} \ \ k)$.
\end{enumerate}
\end{corollary}

\begin{example}
  The induced subgraph $H$ of $\cone_3(\Z)$ with vertex set $V = \{1,2,3,\ldots, 12\}$
  consists of three $4$-cliques. Since $|V| = 12$, $H$ satisfies  $\coneZ(12, 2)$, $\coneZ(4, 4)$, $\coneZ(6, 3)$, $\coneZ(5, 3)$, and $\coneZ(4, 10)$.  
  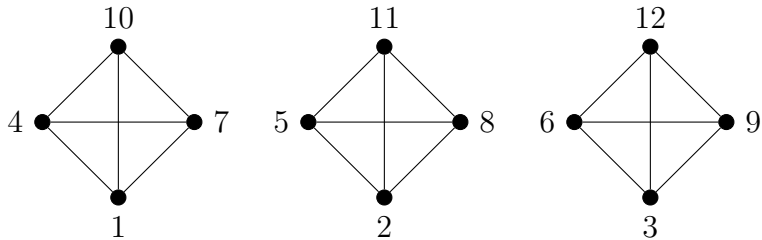
\begin{figure}[H]
    \centering
    \begin{tikzpicture}
  \tikzstyle{vertex}=[circle, fill=black, inner sep = 2,  draw,text=black]

  \node (shift) at (3.5,0) {};
  
  \node[style=vertex, label=below:$1$](p1) at (0,0) {};
  \node[style=vertex, label=left:$4$](p2) at  (-1,1) {};
  \node[style=vertex, label=right:$7$](p3) at  (1,1) {};
  \node[style=vertex, label=above:$10$](p4) at  (0,2) {};

  \draw (p1) -- (p2);
  \draw (p1) -- (p3);
  \draw (p1) -- (p4);
  \draw (p2) -- (p3);
  \draw (p2) -- (p4);
  \draw (p3) -- (p4);

  \node[style=vertex, label=below:$2$](q1) at ($(p1) + (shift)$) {};
  \node[style=vertex, label=left:$5$](q2) at  ($(p2) + (shift)$) {};
  \node[style=vertex, label=right:$8$](q3) at ($(p3) + (shift)$) {};
  \node[style=vertex, label=above:$11$](q4) at ($(p4) + (shift)$) {};

  \draw (q1) -- (q2);
  \draw (q1) -- (q3);
  \draw (q1) -- (q4);
  \draw (q2) -- (q3);
  \draw (q2) -- (q4);
  \draw (q3) -- (q4);
   
  \node[style=vertex, label=below:$3$](r1) at ($(p1) + 2*(shift)$) {};
  \node[style=vertex, label=left:$6$](r2) at  ($(p2) + 2*(shift)$) {};
  \node[style=vertex, label=right:$9$](r3) at ($(p3) + 2*(shift)$) {};
  \node[style=vertex, label=above:$12$](r4) at ($(p4) + 2*(shift)$) {};

  \draw (r1) -- (r2);
  \draw (r1) -- (r3);
  \draw (r1) -- (r4);
  \draw (r2) -- (r3);
  \draw (r2) -- (r4);
  \draw (r3) -- (r4);

\end{tikzpicture}

    \caption{Subgraph $H$ of $\cone_3(\Z)$ induced by the vertices $V = \{1,2,3,\ldots, 12\}$}
    \label{fig:cone-1}
  \end{figure}
\end{example}

\FloatBarrier

\section*{Acknowledgments}
Ayman Badawi is supported by the American University of Sharjah Research Fund (FRG2019): AS1602.

R.~Rissner is supported by the Austrian Science Fund (FWF): P~28466.
\vskip0.1in
We would like to thank the referees for a careful reading of
the paper.

\providecommand{\bysame}{\leavevmode\hbox to3em{\hrulefill}\thinspace}
\providecommand{\MR}{\relax\ifhmode\unskip\space\fi MR }
\providecommand{\MRhref}[2]{%
  \href{http://www.ams.org/mathscinet-getitem?mr=#1}{#2}
}
\providecommand{\href}[2]{#2}

\end{document}